\documentclass[12pt]{article}

\usepackage{amssymb}
\usepackage{amsmath}
\usepackage{graphicx}
\usepackage{pst-all}

\newtheorem{thm}{Theorem}[section]

\newtheorem{lem}[thm]{Lemma}

\newtheorem{obs}[thm]{Observation}

\newenvironment{proof}[1][Proof]{\noindent\textbf{#1.} }{\hfill\rule{1mm}{2mm}}

\makeatletter \@addtoreset{equation}{section} \makeatother

\begin{document}
\title
{Transitivity of Varietal Hypercube Networks
\thanks{The work was
supported by NNSF of China (No. 61272008).} }

\author
{ Li Xiao \quad Jin Cao \quad Jun-Ming Xu\footnote{Corresponding
author, E-mail address: xujm@ustc.edu.cn (J.-M. Xu)}
\\ \\
{\small School of Mathematical Sciences}\\
{\small University of Science and Technology of China}\\
{\small  Wentsun Wu Key Laboratory of CAS}  \\
{\small Hefei, Anhui, 230026, China} }

\date{} \maketitle

\begin{minipage}{140mm}

\begin{center} {\bf Abstract} \end{center}

The varietal hypercube $VQ_n$ is a variant of the hypercube $Q_n$
and has better properties than $Q_n$ with the same number of edges
and vertices. This paper proves that $VQ_n$ is vertex-transitive.
This property shows that when $VQ_n$ is used to model an
interconnection network, it is high symmetrical and obviously
superior to other variants of the hypercube such as the crossed
cube.



\end{minipage}

\vskip0.6cm

\section{Introduction}

We follow~\cite{x03} for graph-theoretical terminology and notation
not defined here. A graph $G=(V,E)$ always means a simple undirected
graph, where $V=V(G)$ is the vertex-set and $E=E(G)$ is the edge-set
of $G$. It is well known that interconnection networks play an
important role in parallel computing/communication systems. An
interconnection network can be modeled by a graph $G=(V, E)$, where
$V$ is the set of processors and $E$ is the set of communication
links in the network.

The hypercube network $Q_n$ has proved to be one of the most popular
interconnection networks since it has a simple structure and has
many nice properties. The varietal hypercubes are proposed by Cheng
and Chuang~\cite{cc94} in 1994 as an attractive alternative to $Q_n$
when they are used to model the interconnection network of a
large-scale parallel processing system.

The $n$-dimensional {\it varietal hypercube} $VQ_n$ is the labeled
graph defined recursively as follows. $VQ_1$ is the complete graph
of two vertices labeled with $0$ and $1$, respectively. Assume that
$VQ_{n-1}$ has been constructed. Let $VQ^0_{n}$ (resp. $VQ^1_{n}$)
be a labeled graph obtained from $VQ_{n-1}$ by inserting a zero
(resp. $1$ ) in front of each vertex-labeling in $VQ_{n-1}$. For
$n>1$, $VQ_n$ is obtained by joining vertices in $VQ^0_{n}$ and
$VQ^1_{n}$, according to the rule: a vertex
$X_n=0x_{n-1}x_{n-2}x_{n-3}\cdots x_2x_1$ in $VQ^0_{n}$ and a vertex
$Y_n=1y_{n-1}y_{n-2}y_{n-3}\cdots y_2y_1$ in $VQ^1_{n}$ are adjacent
in $VQ_n$ if and only if

1. $x_{n-1}x_{n-2}x_{n-3}\cdots x_2x_1=y_{n-1}y_{n-2}y_{n-3}\cdots
y_2y_1$ if $n\ne 3k$, or

2. $x_{n-3}\cdots x_2x_1=y_{n-3}\cdots y_2y_1$ and $(x_{n-1}x_{n-2},
y_{n-1}y_{n-2})\in I$ if $n=3k$, where
$I=\{(00,00),(01,01),(10,11),(11,10)\}$.

Figure~\ref{f1} shows the examples of varietal hypercubes $VQ_n$ for
$n=1, 2, 3$ and $4$.

\vskip30pt

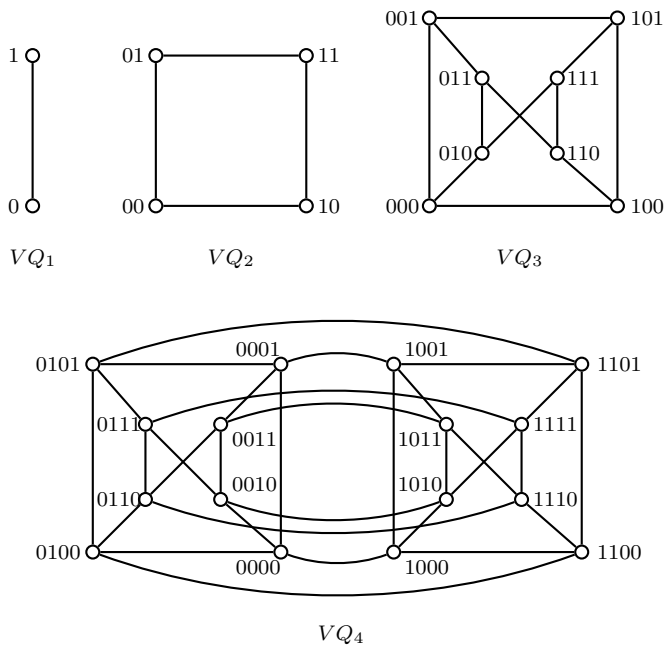
\begin{figure*}[ht]
\begin{pspicture}(-2.3,-.5)(1,3)
\cnode(1,1){.1}{0}\rput(.75,1){\scriptsize0}
\cnode(1,3){.1}{1}\rput(.75,3){\scriptsize1}
\ncline{0}{1}\rput(1,.3){\scriptsize$VQ_1$}
\end{pspicture}
\begin{pspicture}(-.5,-.5)(3,3)
\cnode(1,1){.1}{00}\rput(.7,1){\scriptsize00}
\cnode(1,3){.1}{01}\rput(.7,3){\scriptsize01}
\cnode(3,1){.1}{10}\rput(3.3,1){\scriptsize10}
\cnode(3,3){.1}{11}\rput(3.3,3){\scriptsize11}
\ncline{00}{01}\ncline{01}{11}\ncline{11}{10}\ncline{10}{00}
\rput(2,.3){\scriptsize$VQ_2$}
\end{pspicture}
\begin{pspicture}(-.5,-.5)(3,3)
\cnode(1,1){.1}{000}\rput(.64,1){\scriptsize000}
\cnode(1,3.5){.1}{001}\rput(.64,3.5){\scriptsize001}
\cnode(3.5,1){.1}{100}\rput(3.9,1){\scriptsize100}
\cnode(3.5,3.5){.1}{101}\rput(3.9,3.5){\scriptsize101}
\cnode(1.7,1.7){.1}{010}\rput(1.35,1.7){\scriptsize010}
\cnode(1.7,2.7){.1}{011}\rput(1.35,2.7){\scriptsize011}
\cnode(2.7,1.7){.1}{110}\rput(3.05,1.7){\scriptsize110}
\cnode(2.7,2.7){.1}{111}\rput(3.05,2.7){\scriptsize111}
\ncline{000}{001}\ncline{001}{101}\ncline{101}{100}\ncline{100}{000}
\ncline{010}{011}\ncline{011}{110}\ncline{111}{110}\ncline{111}{010}
\ncline{000}{010}\ncline{001}{011}\ncline{101}{111}\ncline{100}{110}
\rput(2.2,.3){\scriptsize$VQ_3$}
\end{pspicture}

\vskip2pt
\begin{pspicture}(-3.1,0)(8,4)

\cnode(1,1){.1}{0100}\rput(.54,1){\scriptsize0100}%
\cnode(1,3.5){.1}{0101}\rput(.54,3.5){\scriptsize0101}%
\cnode(3.5,1){.1}{0000}\rput(3.2,0.8){\scriptsize0000}%
\cnode(3.5,3.5){.1}{0001}\rput(3.2,3.7){\scriptsize0001}%
\cnode(1.7,1.7){.1}{0110}\rput(1.35,1.7){\scriptsize0110}%
\cnode(1.7,2.7){.1}{0111}\rput(1.35,2.7){\scriptsize0111}%
\cnode(2.7,1.7){.1}{0010}\rput(3.15,1.9){\scriptsize0010}%
\cnode(2.7,2.7){.1}{0011}\rput(3.15,2.5){\scriptsize0011}

\ncline{0000}{0001}\ncline{0001}{0101}\ncline{0101}{0100}\ncline{0100}{0000}
\ncline{0010}{0011}\ncline{0011}{0110}\ncline{0111}{0110}\ncline{0111}{0010}
\ncline{0000}{0010}\ncline{0001}{0011}\ncline{0101}{0111}\ncline{0100}{0110}

\cnode(5,1){.1}{1000}\rput(5.44,0.8){\scriptsize1000}
\cnode(5,3.5){.1}{1001}\rput(5.44,3.7){\scriptsize1001}
\cnode(7.5,1){.1}{1100}\rput(8.,1){\scriptsize1100}
\cnode(7.5,3.5){.1}{1101}\rput(8.0,3.5){\scriptsize1101}
\cnode(5.7,1.7){.1}{1010}\rput(5.35,1.9){\scriptsize1010}
\cnode(5.7,2.7){.1}{1011}\rput(5.35,2.5){\scriptsize1011}
\cnode(6.7,1.7){.1}{1110}\rput(7.15,1.7){\scriptsize1110}
\cnode(6.7,2.7){.1}{1111}\rput(7.15,2.7){\scriptsize1111}

\ncline{1000}{1001}\ncline{1001}{1101}\ncline{1101}{1100}\ncline{1100}{1000}
\ncline{1010}{1011}\ncline{1011}{1110}\ncline{1111}{1110}\ncline{1111}{1010}
\ncline{1000}{1010}\ncline{1001}{1011}\ncline{1101}{1111}\ncline{1100}{1110}

\nccurve[angleA=20,angleB=160]{0101}{1101}
\nccurve[angleA=20,angleB=160]{0111}{1111}
\nccurve[angleA=-20,angleB=-160]{0110}{1110}
\nccurve[angleA=-20,angleB=-160]{0100}{1100}
\nccurve[angleA=20,angleB=160]{0001}{1001}
\nccurve[angleA=20,angleB=160]{0011}{1011}
\nccurve[angleA=-20,angleB=-160]{0010}{1010}
\nccurve[angleA=-20,angleB=-160]{0000}{1000}
\rput(4.3,-0.1){\scriptsize$VQ_4$}

\end{pspicture}

\caption{\label{f1}{\footnotesize \ The varietal hypercubes $VQ_1,
VQ_2, VQ_3$ and $VQ_4$}}
\end{figure*}

Like $Q_n$, $VQ_n$ is an $n$-regular graph with $2^n$ vertices and
$n2^{n-1}$ edges, and has many properties similar or superior to
$Q_n$. For example, the connectivity and restricted connectivity of
$VQ_n$ and $Q_n$ are the same (see Wang and Xu~\cite{wx09}), while,
all the diameter and the average distance, fault-diameter and
wide-diameter of $VQ_n$ are smaller than that of the hypercube (see
Cheng and Chuang~\cite{cc94}, Jiang {\it et al.}~\cite{jhl10}). Very
recently, Cao {\it et al}.~\cite{cxx12} have shown that $VQ_n$ has
better pancyclicity and panconnectvity than $Q_n$.

An {\it automorphism} of a graph $G$ is a permutation $\sigma$ on
$V(G)$ satisfying the adjacency-preserving condition
$$
xy\in E(G) \Leftrightarrow \sigma(x)\sigma(y)\in E(G).
$$
Under the operation of composition, the set of all automorphisms of
$G$ forms a group, denoted by $Aut(G)$ and referred to as the {\it
automorphism group} of $G$.

A graph $G$ is {\it vertex-transitive} if for any given two vertices
$x$ and $y$ in $G$ there is some $\sigma\in Aut(G)$ such that
$y=\sigma(x)$. A graph $G$ is called to be {\it edge-transitive} if
for any two given edges $xy$ and $uv$ of $G$ there is some $\sigma
\in Aut(G)$ such that $uv=\sigma(x)\sigma(y)$.

It has known that $Q_n$ is vertex-transitive and edge-transitive,
the folded hypercube $FQ_n$ is vertex-transitive (see Ma and
Xu~\cite{mx10}, the crossed hypercube $CQ_n$ is not
vertex-transitive for $n\geqslant 5$ (see Kulasinghe and
Bettayeb~\cite{kb95}). However, transitivity of some variant of the
hypercube has not been investigated. In this paper, we consider
transitivity of $VQ_n$. Choose three vertices $X=0101, Y=1101$ and
$Z=0001$ in $VQ_4$ (see Figure~\ref{f1}), the edge $XZ$ is contained
in a cycle of length $5$, but the edge $XY$ is not. This fact shows
that there is no $\sigma \in Aut(VQ_4)$ such that
$XZ=\sigma(X)\sigma(Y)$. However, we can show that $VQ_n$ is
vertex-transitive. This property shows that when $VQ_n$ is used to
model an interconnection network, it is high symmetrical and
obviously superior to other variants of the hypercube, such as the
crossed cube $CQ_n$.

\section{Main results}

For convenience, we express $VQ_n$ as $VQ^0_{n}\odot VQ^1_{n}$,
where $VQ^0_{n}\cong VQ^1_{n}\cong VQ_{n-1}$. Edges between
$VQ^0_{n}$ and $VQ^1_{n}$ are called {\it $n$-transversal edges},
the edges of Type $2$ are called {\it crossing edges} when
$(x_{n-1}x_{n-2}, y_{n-1}y_{n-2})\in \{(10,11),(11,10)\}$, and all
the other edges are called {\it normal edges}. For a given position
integer $n$, let $I_{n}=\{1,2,\ldots, n\}$ and let
$$
 V_n=\{x_{n}\cdots x_{2}x_1|\ x_i\in \{0,1\}, i\in
 I_{n}\}.
 $$
Clearly, $V(VQ_n)=V_n$. For a given $X_n=x_{n}\cdots x_{2}x_1\in
V_n$, let $X_i=x_i\cdots x_2x_1$. For $b\in\{0,1\}$, let
$\bar{b}=\{0,1\}\setminus\{b\}$. By definitions, we immediately
obtain the following simple observation.

\begin{obs}\label{obs2.1}\
Let $VQ_n=VQ^0_{n}\odot VQ^1_{n}$ and $X_nY_n$ be an $n$-transversal
edge in $VQ_n$, where $X_n\in VQ^0_{n}$ and $Y_n\in VQ^1_{n}$. For
$n\ge 3$, if $X_n=0abx_{n-3}\cdots x_1$, then $Y_n=1a'b'X_{n-3}$,
where $ab=a'b'$ if $X_nY_n$ is a normal edge, and
$(ab,a'b')=(1b,1\bar{b})$ if $X_nY_n$ is a crossing edge.
\end{obs}

\begin{lem}\label{lem2.2}\ Define a mapping
 \begin{equation}\label{e2.1}
 \begin{array}{rl}
 \sigma_1: &  V_n \, \to\ V_n\\
    & X_n \mapsto\ \bar{x}_nX_{n-1}.
 \end{array}
 \end{equation}
Then $\sigma_1\in Aut(VQ_{n})$.
\end{lem}

\begin{proof}\ Clearly, $\sigma_1$ is a permutation on $V_n$.
We only need to show that $\sigma_1$ preserves the adjacency of
vertices in $VQ_n$. To the end, let $X_nY_n\in E(VQ_n)$, where
$X_n=x_nx_{n-1}\cdots x_{2}x_1$ and $Y_n=y_ny_{n-1}\cdots y_{2}y_1$.
Without loss of generality, assume $x_n=0$. Then
$\sigma_1(X_n)=1X_{n-1}$ in $VQ^1_n$.

If $y_n=0$, then $X_{n}Y_{n}\in E(VQ^{0}_{n})$, and so
$X_{n-1}Y_{n-1}\in E(VQ_{n-1})$. Since $\sigma_1(Y_n)=1Y_{n-1}$ in
$VQ^1_n$, $\sigma_1(X_n)\sigma_1(Y_n)$ is an edge in $VQ^{1}_{n}$,
and so in $VQ_{n}$.

We now assume $y_n=1$. Then $Y_n$ is in $VQ^1_n$ and $\sigma_1(Y_n)$
is in $VQ^0_n$. Thus $X_nY_n$ is an $n$-transversal edge.

If $X_nY_n$ is a normal edge, then $X_{n-1}=Y_{n-1}$, and
$\sigma_1(X_n)\sigma_1(Y_n)=Y_nX_n\in E(VQ_n)$ clearly.

If $X_nY_n$ is a crossing edge, then $(x_{n-1}x_{n-2},
y_{n-1}y_{n-2})=(1x_{n-2},1\bar{x}_{n-2})$. That is,
$X=01x_{n-2}X_{n-3}$ and $Y=11\bar{x}_{n-2}X_{n-3}$, and so
 $$
 \begin{array}{rl}
 &\sigma_1(X_n)=11x_{n-2}X_{n-3}\ \ {\rm and}\\
 &\,\sigma_1(Y_n)=01\bar{x}_{n-2}X_{n-3},
 \end{array}
 $$
which shows that $\sigma_1(X_n)\sigma_1(Y_n)$ is a crossing edge in
$VQ_n$.

The lemma follows.\end{proof}

\begin{lem}\label{lem2.3}\
For a given $\phi\in Aut(VQ_{n-3})$, define a mapping $\varphi_i$
from $V_{n-1}$ to $V_{n-1}$ for each $i=0,1,2,3$ subjected to
 \begin{equation}\label{e2.2}
 \begin{array}{rl}
  & \varphi_0(X_{n-1})=x_{n-1}{x}_{n-2}\phi(X_{n-3})\\
  & \varphi_1(X_{n-1})=x_{n-1}\bar{x}_{n-2}\phi(X_{n-3})\\
  & \varphi_2(X_{n-1})=\bar{x}_{n-1}{x}_{n-2}\phi(X_{n-3})\\
  & \varphi_3(X_{n-1})=\bar{x}_{n-1}\bar{x}_{n-2}\phi(X_{n-3}).
 \end{array}
 \end{equation}
If $n=3k$, then $\varphi_i\in Aut(Q_{n-1})$ for each $i=0,1,2,3$.
\end{lem}

\begin{proof}\
We first show $\varphi_1,\varphi_2\in Aut(Q_{n-1})$. To the end, we
only need to show that both $\varphi_1$ and $\varphi_2$ preserve the
adjacency of vertices in $VQ_{n-1}$. Let $X_{n-1}Y_{n-1}\in
E(VQ_{n-1})$, where $X_{n-1}=x_{n-1}\cdots x_{2}x_1$ and
$Y_{n-1}=y_{n-1}\cdots y_{2}y_1$. We want to show that
$\varphi_i(X_{n-1})\varphi_i(Y_{n-1})\in E(VQ_{n-1})$ for each
$i=1,2$.

Without loss of generality, assume $x_{n-1}x_{n-2}=0b$, where
$b\in\{0,1\}$. Then $\varphi_1(X_{n-1})=0\bar{b}\phi(X_{n-3})$ is in
$VQ^0_{n-1}$ and $\varphi_2(X_{n-1})=1{b}\phi(X_{n-3})$ is in
$VQ^1_{n-1}$. Note $n-1\ne 3k$ since $n=3k$. There are two cases
according as $y_{n-1}=0$ or $1$.

\vskip6pt

{\bf Case 1.}\ $y_{n-1}=0$. Then $X_{n-1}Y_{n-1}\in E(VQ^0_{n-1})$.

\begin{enumerate}
  \item
If $y_{n-2}=b$, then $X_{n-1}Y_{n-1}\in E(VQ^{0b}_{n-1})$, where
$VQ^{0b}_{n-1}$ denotes a subgraph of $VQ_{n-1}$ obtained from
$VQ_{n-3}$ by inserting two digits $0b$ in front of each
vertex-labeling in $VQ_{n-3}$, which is isomorphic to $VQ_{n-3}$.
Thus, $X_{n-3}Y_{n-3}\in E(VQ_{n-3})$. Since $\phi\in
Aut(VQ_{n-3})$, $\phi(X_{n-3})\phi(Y_{n-3})\in E(VQ_{n-3})$. Thus,
two vertices $0b\phi(X_{n-3})$ and $0b\phi(Y_{n-3})$ are adjacent in
$VQ^{0b}_{n-1}$, and so two vertices
 $$
 \begin{array}{rl}
 &\varphi_1(X_{n-1})=0\bar{b}\phi(X_{n-3})\ {\rm and}\\
 &\, \varphi_1(Y_{n-1})=0\bar{b}\phi(Y_{n-3})
 \end{array}
 \ \ \left({\rm resp.}
\begin{array}{rl}
 &\varphi_2(X_{n-1})=1{b}\phi(X_{n-3})\ {\rm and}\\
 &\, \varphi_2(Y_{n-1})=1{b}\phi(Y_{n-3})
 \end{array}\right)
 $$
are adjacent in $VQ^{0\bar{b}}_{n-1}$ (resp. $VQ^{1{b}}_{n-1}$),
that is, $\varphi_i(X_{n-1})\varphi_i(Y_{n-1})\in E(VQ_{n-1})$ for
each $i=1,2$.

  \item
If $y_{n-2}=\bar{b}$, then $X_{n-1}Y_{n-1}$ is an
$(n-2)$-transversal normal edge in $VQ^0_{n-1}$ since $n-2\ne 3k$.
Thus, $X_{n-3}=Y_{n-3}$, and so $\phi(X_{n-3})=\phi(Y_{n-3})$. Since
 $$
 \begin{array}{rl}
 &\varphi_1(X_{n-1})=0\bar{b}\phi(X_{n-3})\ {\rm and}\\
 &\,\varphi_1(Y_{n-1})=0b\phi(X_{n-3}),
\end{array}
 \ \ \left({\rm resp.}
\begin{array}{rl}
 &\varphi_2(X_{n-1})=1{b}\phi(X_{n-3})\ {\rm and}\\
 &\,\varphi_2(Y_{n-1})=1\bar{b}\phi(X_{n-3}),
\end{array}\right)
 $$
$\varphi_1(X_{n-1})\varphi_1(Y_{n-1})$ (resp.
$\varphi_2(X_{n-1})\varphi_2(Y_{n-1})$) is also an
$(n-2)$-transversal normal edge in $VQ^0_{n-1}$ (resp.
$VQ^1_{n-1}$).
\end{enumerate}

\vskip6pt

{\bf Case 2.}\ $y_{n-1}=1$.

Then $X_{n-1}Y_{n-1}$ is an $(n-1)$-transversal normal edge in
$VQ_{n-1}$ since $n-1\ne 3k$, and so $X_{n-2}=Y_{n-2}$. Since
$\phi(X_{n-3})=\phi(Y_{n-3})$, we have that
 $$
 \begin{array}{rl}
 &\varphi_1(X_{n-1})=0\bar{b}\phi(X_{n-3})\ {\rm and}\\
 &\,\varphi_1(Y_{n-1})=1\bar{b}\phi(X_{n-3}),
\end{array}
 \ \ \left({\rm resp.}
\begin{array}{rl}
 &\varphi_2(X_{n-1})=1{b}\phi(X_{n-3})\ {\rm and}\\
 &\,\varphi_2(Y_{n-1})=0{b}\phi(X_{n-3}),
\end{array}\right)
 $$
which implies that $\varphi_i(X_{n-1})\varphi_i(Y_{n-1})$ is an
$(n-1)$-transversal edge in $VQ_{n-1}$ for each $i=1,2$.

\vskip6pt

Thus, $\varphi_i\in Aut(Q_{n-1})$ for each $i=1,2$. Since
$\varphi_3=\varphi_1\varphi_2$ and $\varphi_0=\varphi_3^2$, we have
$\varphi_3, \varphi_0\in Aut(Q_{n-1})$ immediately. The lemma
follows.
\end{proof}

\begin{lem}\label{lem2.4}\
For a given $\varphi\in Aut(VQ_{n-1})$, define a mapping
 \begin{equation}\label{e2.3}
 \begin{array}{rl}
 \sigma_0: &  V_{n}\, \to\ V_{n}\\
 & X_{n} \mapsto\ x_{n}\varphi(X_{n-1}).
 \end{array}
 \end{equation}
When $n\ne 3k$, $\sigma_0\in Aut(VQ_n)$ for any  $\varphi\in
Aut(VQ_{n-1})$. When $n=3k$, if
 \begin{equation}\label{e2.4}
 \varphi=\left\{
 \begin{array}{l}
  \varphi_0\ {\rm or}\ \varphi_1\ {\rm or}\\
  \varphi_2\ {\rm when}\ x_n=1\ {\rm and}\ \varphi_3\ {\rm when}\ x_n=0\ {\rm or}\\
  \varphi_2\ {\rm when}\ x_n=0\ {\rm and}\ \varphi_3\ {\rm when}\ x_n=1,
  \end{array}\right.
 \end{equation}
where $\varphi_i$ is defined in (\ref{e2.2}) for each $i=0,1,2,3$,
then $\sigma_0\in Aut(VQ_n)$.
\end{lem}

\begin{proof}\
It is easy to see that $\sigma_0$ is a permutation on $V_n$. We show
$\sigma_0\in Aut(VQ_n)$. To the end, we only need to prove that
$\sigma_0$ preserves the adjacency of vertices in $VQ_{n}$. Let
$VQ_n=VQ^0_n\odot VQ^1_n$ and let $X_n=x_nx_{n-1}\cdots x_{2}x_1$
and $Y_n=y_ny_{n-1}\cdots y_{2}y_1$ be any two adjacent vertices in
$VQ_n$. Without loss of generality, let $x_n=0$. Then $X_n$ is in
$VQ^0_n$. There are two cases according as $y_{n}=0$ or $1$.

\vskip6pt

{\bf Case 1.}\ $y_n=0$.

Since both $X_n$ and $Y_n$ are in $VQ^0_n$, for any $\varphi\in
Aut(VQ_{n-1})$, by (\ref{e2.3}) we have that
 $$
 \begin{array}{rl}
 &\sigma_0(X_n)= 0\varphi(X_{n-1})\ \ {\rm and}\\
 &\,\sigma_0(Y_n)= 0\varphi(Y_{n-1}).
 \end{array}
 $$
Since $X_nY_n\in E(VQ^0_{n})$, $X_{n-1}Y_{n-1}\in E(VQ_{n-1})$, and
so $\varphi(X_{n-1})\varphi(Y_{n-1})\in E(VQ_{n-1})$ and, hence,
$\varphi(X_n)\varphi(Y_n)\in E(VQ^0_{n})\subset E(VQ_n)$.

\vskip6pt

{\bf Case 2.}\ $y_n=1$.

In this case, $Y_n$ is in $VQ^1_n$ and $X_nY_n$ is an
$n$-transversal edge in $VQ_n$, which is either a normal edge or a
crossing edge.

\begin{enumerate}
  \item \ $X_nY_n$ is a crossing edge.

In this subcase, $n=3k$, $X_{n-3}=Y_{n-3}$, and
$X_n=01x_{n-2}X_{n-3}$ and $Y_n=11\bar{x}_{n-2}X_{n-3}$ by
Observation~\ref{obs2.1}.

Since $X_{n-1}=1x_{n-2}X_{n-3}$ and $Y_{n-1}=1\bar{x}_{n-2}X_{n-3}$,
$X_{n-1}Y_{n-1}$ is an $(n-2)$-dimensional normal edge in
$VQ^1_{n-1}$ since $n-1\ne 3k$. Thus, for any $\varphi\in
Aut(VQ_{n-1})$, $\varphi(X_{n-1})\varphi(Y_{n-1})$ is an
$(n-2)$-dimensional normal edge in $VQ^1_{n-1}$. Without loss of
generality, let $\varphi(X_{n-1})=10U_{n-3}$. Then
$\varphi(Y_{n-1})=11U_{n-3}$, and so,
$$
\begin{array}{rl}
 &\sigma_0(X_{n})=x_n\varphi(X_{n-1})=010U_{n-3}\\
 &\,\sigma_0(Y_{n})\,=y_n\,\varphi(Y_{n-1})\,=111U_{n-3}.
\end{array}
$$
Thus, $\sigma_0(X_{n})\sigma_0(Y_{n})$ is an $n$-dimensional
crossing edge in $VQ_{n}$.

  \item\  $X_nY_n$ is a normal edge.

In this subcase, $X_{n-1}=Y_{n-1}$, and so
$\varphi(X_{n-1})=\varphi(Y_{n-1})$ for any $\varphi\in
Aut(VQ_{n-1})$. By (\ref{e2.3}), we have that
 $$
 \begin{array}{rl}
 &\sigma_0(X_n)=0\varphi(X_{n-1})\ \ {\rm and}\\
 &\,\sigma_0(Y_n)=1\varphi(Y_{n-1}).
 \end{array}
 $$

If $n\ne 3k$, then $\sigma_0(X_n)\sigma_0(Y_n)\in E(VQ_n)$ is a
normal edge. Assume now $n=3k$. Then
$\{x_{n-1}x_{n-2},y_{n-1}y_{n-2})\in\{(00,00),(01,01)\}$.

\qquad If $\varphi=\varphi_0$ or $\varphi_1$, then
 $$
 \left\{\begin{array}{l}
 \sigma_0(X_{n-1})=0\varphi_0(X_{n-1})=0x_{n-1}{x}_{n-2}\phi(X_{n-3}),\ {\rm
 and}\\
 \,\sigma_0(Y_{n-1})\,=1\varphi_0(Y_{n-1})=1x_{n-1}{x}_{n-2}\phi(X_{n-3})
 \end{array}\right.\\
 $$
 or
 $$
 \left\{\begin{array}{l}
 \sigma_0(X_{n})=0\varphi_1(X_{n-1})=0x_{n-1}\bar{x}_{n-2}\phi(X_{n-3}),\ {\rm
 and}\\
 \,\sigma_0(Y_{n})\,=1\varphi_1(Y_{n-1})=1x_{n-1}\bar{x}_{n-2}\phi(X_{n-3}).
 \end{array}\right.
 $$
Thus, $\sigma_0(X_{n-1})\sigma_0(Y_{n-1})$ is an $n$-dimensional
normal edge in $VQ_n$.

\qquad If $\varphi=\varphi_2$ when $x_n=1$ and $\varphi=\varphi_3$
when $x_n=0$, then
 $$
 \left\{\begin{array}{l}
 \sigma_0(X_{n})=0\varphi_3(X_{n-1})=0\bar{x}_{n-1}\bar{x}_{n-2}\phi(X_{n-3}),\ {\rm
 and}\\
 \,\sigma_0(Y_{n})\,=1\varphi_2(Y_{n-1})=1\bar{x}_{n-1}{x}_{n-2}\phi(X_{n-3}).
 \end{array}\right.
 $$

\qquad If $\varphi=\varphi_2$ when $x_n=0$ and $\varphi=\varphi_3$
when $x_n=1$, then
 $$
 \left\{\begin{array}{l}
 \sigma_0(X_{n})=0\varphi_2(X_{n-1})=0\bar{x}_{n-1}{x}_{n-2}\phi(X_{n-3}),\ {\rm
 and}\\
 \,\sigma_0(Y_{n})\,=1\varphi_3(Y_{n-1})=1\bar{x}_{n-1}\bar{x}_{n-2}\phi(X_{n-3}).
 \end{array}\right.
 $$

Since $\{\bar{x}_{n-1}\bar{x}_{n-2},\bar{x}_{n-1}{x}_{n-2})\in
\{(11,10),(10,11)\}$, $\sigma_0(X_{n})$ and $\sigma_0(Y_{n})$ are
linked by an $n$-dimensional crossing edge in $VQ_n$.

\end{enumerate}

Thus, we have proved that $\sigma_0$ preserves the adjacency of
vertices in $VQ_{n}$, and so $\sigma_0\in Aut(VQ_n)$. The lemma
follows.
\end{proof}

\begin{thm}\label{thm}\
$VQ_n$ is vertex-transitive for any $n\geqslant 1$.
\end{thm}

\begin{proof}\
We proceed by induction on $n\geqslant 1$. The conclusion is true
for each $n=1,2$ clearly. Since $VQ_3$ is isomorphic to a Cayley
graph $C(Z_8,\{1,4,7\})$ (see Huang and Xu~\cite{hx05}), the
conclusion is also true for $n=3$.

Assume the induction hypothesis for any positive integer fewer than
$n$ with $n\geqslant 4$, that is, for any $i$ with $1\leqslant
i\leqslant n-1$, $VQ_{i}$ is vertex-transitive.

Let $VQ_n=VQ^0_n\odot VQ^1_n$ be an $n$-dimensional varietal
hypercube with $n\geqslant 4$. To prove that $VQ_n$ is
vertex-transitive, we need to prove that for any two vertices $X_n$
and $Y_n$ in $VQ_n$ there is some $\sigma\in Aut(VQ_n)$ such that
$\sigma(X_n)=Y_n$. To the end, let $X_n=x_nx_{n-1}\cdots x_{2}x_1$
and $Y_n=y_ny_{n-1}\cdots y_{2}y_1$ be any two vertices in $VQ_n$.
By the induction hypothesis, there is some $\varphi\in
Aut(VQ_{n-1})$ such that $\varphi(X_{n-1})=Y_{n-1}$.

We now construct a $\sigma\in Aut(VQ_n)$ with $\sigma(X_n)=Y_n$.
Without loss of generality, let $x_n=0$. Then $X_n$ is in $VQ^0_n$.
There are two cases.

\vskip6pt

{\bf Case 1.}\ $y_n=0$.

In this case, both $X_n$ and $Y_n$ are in $VQ^0_n$. Assume first
$n\ne 3k$. By the induction hypothesis, there is some $\varphi\in
Aut(VQ_{n-1})$ such that $\varphi(X_{n-1})=Y_{n-1}$. Let
$\sigma=\sigma_0$, where $\sigma_0$ is defined in (\ref{e2.3}). By
Lemma~\ref{lem2.4}, $\sigma\in Aut(VQ_n)$, and
 $$
 \sigma(X_n)=0\varphi(X_{n-1})=0Y_{n-1}=Y_n.
 $$

Assume $n=3k$ below. Since $Y_{n-1}$ is one of the following four
forms:
 $$
 \begin{array}{rl}
  & Y_{n-1}=x_{n-1}{x}_{n-2}Y_{n-3}\\
  & Y_{n-1}=x_{n-1}\bar{x}_{n-2}Y_{n-3}\\
  & Y_{n-1}=\bar{x}_{n-1}{x}_{n-2}Y_{n-3}\\
  & Y_{n-1}=\bar{x}_{n-1}\bar{x}_{n-2}Y_{n-3}.
 \end{array}
 $$

By the induction hypothesis, there is some $\phi\in Aut(VQ_{n-3})$
such that $\phi(X_{n-3})=Y_{n-3}$. Let $\sigma=\sigma_0$, where
$\sigma_0\in Aut(VQ_n)$ is defined in (\ref{e2.4}).
 $$
 \sigma_0(X_{n})=\left\{
 \begin{array}{l}
  x_n\varphi_0(X_{n-1})=0x_{n-1}{x}_{n-2}\phi(X_{n-3})=Y_n\ {\rm if}\ Y_{n-1}=x_{n-1}{x}_{n-2}Y_{n-3}\\
  x_n\varphi_1(X_{n-1})=0x_{n-1}\bar{x}_{n-2}\phi(X_{n-3})=Y_n\ {\rm if}\ Y_{n-1}=x_{n-1}\bar{x}_{n-2}Y_{n-3}\\
  0\varphi_2(X_{n-1})=0\bar{x}_{n-1}{x}_{n-2}\phi(X_{n-3})=Y_n\ {\rm if}\ Y_{n-1}=\bar{x}_{n-1}{x}_{n-2}Y_{n-3}\\
  0\varphi_3(X_{n-1})=0\bar{x}_{n-1}\bar{x}_{n-2}\phi(X_{n-3})=Y_n\ {\rm
  if}\ Y_{n-1}=\bar{x}_{n-1}\bar{x}_{n-2}Y_{n-3}.
    \end{array}\right.
 $$

{\bf Case 2.}\ $y_n=1$.

In this case, $Y_n$ is in $VQ^1_n$. Consider $\sigma_1$ and
$\sigma_0$, defined in (\ref{e2.1}) and (\ref{e2.3}) respectively.
If $X_{n-1}=Y_{n-1}$, let $\sigma=\sigma_0$, then $\sigma\in
Aut(VQ_n)$, and $\sigma(X_n)=\sigma_0(X_n)=1X_{n-1}=Y_n$. If
$X_{n-1}\ne Y_{n-1}$, let $\sigma=\sigma_1\sigma_0$, then
$\sigma=\sigma_1\sigma_0\in Aut(VQ_n)$, and
 $$
 \sigma(X_n)=\sigma_1\sigma_0(X_n)=\sigma_1(x_n\varphi(X_{n-1}))=\sigma_1(0Y_{n-1})=1Y_{n-1}=Y_n.
 $$

Thus, we have proved that for any two vertices $X_n$ and $Y_n$ in
$VQ_n$ there is a $\sigma\in Aut(VQ_n)$ such that $\sigma(X_n)=Y_n$,
and so $VQ_n$ is vertex-transitive. The theorem follows.
\end{proof}

\end{document}